\documentclass[11pt,leqno]{amsart}
\usepackage{amsmath,amssymb,amsthm}
\usepackage{hyperref}

\usepackage{tikz-cd}

\usepackage{hyperref}

\usepackage{bbm}

\newcommand{\lspan}{\operatorname{span}}
\newcommand{\clspan}{\overline{\operatorname{span}}}

\renewcommand{\geq}{\geqslant}
\renewcommand{\leq}{\leqslant}

\newcommand{\norm}[1]{\left\Vert#1\right\Vert}

\newcommand{\Lip}{{\mathrm{Lip}}_0}

\newcommand{\dens}{\operatorname{dens}}

\newcommand{\spann}{\operatorname{span}}

\newtheorem{theorem}{Theorem}[section]

\newtheorem{proposition}[theorem]{Proposition}
\newtheorem{corollary}[theorem]{Corollary}
\theoremstyle{definition}
\newtheorem{definition}[theorem]{Definition}

\theoremstyle{remark}
\newtheorem{remark}[theorem]{Remark}

\numberwithin{equation}{section}

\def\fnote#1{\footnote}

\def\ignora#1{}
\def\n3#1{\left\vert  \! \left\vert \! \left\vert \, #1 \, \right\vert \!
  \right\vert \! \right\vert }

\newcommand{\iten}{\ensuremath{\widehat{\otimes}_\varepsilon}}
\newcommand{\pten}{\ensuremath{\widehat{\otimes}_\pi}}

\renewcommand{\leq}{\le}

\usepackage{accents}
\usepackage[most]{tcolorbox}
\usepackage{comment}

\let\emptyset\varnothing

\newcommand{\ra}{\longrightarrow}

\newcommand{\eps}{\varepsilon}

\begin{document}

\author{ Abraham Rueda Zoca }\address{Universidad de Granada, Facultad de Ciencias. Departamento de An\'{a}lisis Matem\'{a}tico, 18071-Granada
(Spain)} \email{ abrahamrueda@ugr.es}
\urladdr{\url{https://arzenglish.wordpress.com}}



\title{Transfinite octahedrality in spaces of operators}

\begin{abstract}
We study necessary and sufficient conditions on a Banach space $X$ such that $L(Y,X)$ is transfinite (rigid) octahedral for any Banach space $Y$. We present several examples of such $X$. As application of our results, we find the first example in the literature of Banach space which is rigid octahedral but fails any transfinite version of octahedrality. We also improve some previous results about octahedrality in ultrapower spaces.
\end{abstract}

\maketitle

\section{Introduction}

One of the most powerful tools in Hilbert spaces theory is the natural notion of orthogonality of vectors, which permits to obtain a number of results which are exclusive of those Banach spaces with a scalar product (e.g. orthogonal projections, complementation of closed subspaces, self duality etc.). Because of this reason, several notions of orthogonality of vectors have appeared in the literature of Banach spaces like James orthogonality \cite{jam} or geometric notions which are inspired by the natural orthogonality induced by the usual basis in spaces $\ell_p$ and $c_0$ (see \cite{acllrz23,amcrz23,brss26,cll23,gk89}). In this note we will focus on the particular case of the $\ell_1$ orthogonality, which drives us to the notions of octahedrality.

Given a Banach space $X$ and an infinite cardinal $\kappa$, we say that $X$ is \textit{$\kappa$-octahedral} if for any $\kappa'<\kappa$, any collection $\{x_\alpha: \alpha<\kappa'\}\subseteq S_X$ and any $\varepsilon>0$ there exists $x\in S_X$ such that
$$\Vert x_\alpha+x\Vert\geq 2-\varepsilon\ \mbox{ holds for all }\alpha<\kappa'.$$
If $\varepsilon=0$ can be raiched we say that $X$ is \textit{$\kappa$-rigid octahedral}.

If $\kappa=\omega_0$, we will simply say that $X$ is (rigid) octahedral. The notion of octahedral norm goes back to the unpublished preprint \cite{godmau87}, where it was proved that a separable Banach space $X$ contains an isomorphic copy of $\ell_1$ if, and only if, there exists an equivalent octahedral renorming on $X$. Later, the result was proved in \cite{god89} without the separability assumption. Very recently, the publication \cite{blr14}, where the connection between octahedral norms and the big slice phenomena is first proved, resulted in a boost in the study of octahedrality and it is recently a hot area in the study of Banach spaces (as a sample of this c.f. e.g. \cite{abra15,hlp,lalo21,laru20,pr18}). For uncountable cardinals, the notions of $\kappa$ (rigid) octahedrality goes back to the paper \cite{cll23}, where these properties arose in connection with ball covering properties. Thanks to the results of \cite{amcrz23,cll23}, it was proved that a Banach space $X$ contains an isomorphic copy of $\ell_1(\kappa)$ if and only if there exists an equivalent renorming on $X$ to be $\kappa$ rigid octahedral. More results about transfinite (rigid) octahedrality appeared in \cite{acllrz23,brss26,cll22}. It is known that all the notions of octahedrality are different eachother. To be more precise, it is known that given an infinite cardinal $\kappa$, then $\kappa$-octahedrality does not imply $\kappa$-rigid octahedrality, whereas $\kappa$-rigid octahedrality does not imply $\kappa^+$ octahedrality (and examples can be found in \cite{cll23}) with a single exception: it is not known, to the best of the authors' knowledge, any example of $\omega_0$-rigid octahedral space which is not $\omega_1$-octahedral. We provide such a counterexample in Remark~\ref{remark:rigidohnosepara}.

In this paper we will deepen in the study of transfinite (rigid) octahedrality in spaces of operators. Apart from being a natural question, as spaces of operators are classical Banach spaces, the study of octahedrality in spaces of operators have produced applications to the general theory of octahedrality in Banach spaces. To mention two of them, in \cite[Theorem 3.2]{lrz21} it was shown that the space of compact operators $X:=K(\ell_2(2^{\dens(L_1([0,1])^{**})}),L_1([0,1]))$ produces an example of octahedral space for which there exists no element $x^{**}\in X^{**}\setminus\{0\}$ such that
$$\Vert x+x^{**}\Vert=\Vert x\Vert+\Vert x^{**}\Vert\ \forall x\in X,$$
answering an environment question from \cite{gk89}. Another application was obtained in \cite[Theorem 8.2]{amcrz2022}, where the fact that $K(\ell_2,X)$ is octahedral whenever $X$ is octahedral was employed to prove that, given any separable octahedral space $X$, it follows that the (non-empty) set
$$\{x^{**}\in X^{**}: \Vert x+x^{**}\Vert=\Vert x\Vert+\Vert x^{**}\Vert \ \forall x\in X\},$$
which is not a vector space, always contains a subspace of $X^{**}$ which is isometrically isomorphic to $\ell_2$. 

The literature on octahedrality in spaces of operators is vast \cite{blrope,hlp2,llr,llr2,meru25,rue23}. It is known that given two Banach spaces $X$ and $Y$, if $X^*$ and $Y$ are octahedral then so is any subspace of $L(X,Y)$ containing the finite rank operators \cite[Theorem 3.5]{blrope}, but the result is not true if one assumes octahedrality in either $X^*$ or $Y$ \cite[Theorem 3.8]{llr2}. Observe that the same conclussion is obtained for $<\kappa$ octahedrality for any uncountable cardinal $\kappa$ in \cite[Proposition 5.4 and Remark 5.2]{cll22}. Moreover, those Banach spaces $Y$ such that $L(X,Y)$ is octahedral for every $X$ are characterised in \cite{rue23} in terms of a condition involving the finite representability of every Banach space in $Y$ (see \cite[Theorem 3.7]{rue23}). Indeed, it was proved in \cite[Theorem 3.7]{rue23} that, given a Banach space $X$, the following conditions are equivalent:
\begin{enumerate}
    \item For every Banach space $Y$, every subspace of $L(Y,X)$ containing the finite rank operators is octahedral,
    \item given any finite set of vectors $\{x_1,\ldots, x_n\}\subseteq S_X$, any $k\in\mathbb N$ and any $\varepsilon>0$ there exists $T:\ell_\infty^k\longrightarrow X$ satisfying that
$$\Vert x_i+T(v)\Vert\geq (1-\varepsilon)(1+\Vert v\Vert)$$
holds for every $v\in \ell_\infty^k$.
\end{enumerate}  
The above condition was defined in \cite{rue23} as ``$X$ is universally octahedral''. After this, a strenghtening of universal octahedrality is defined in \cite[Definition 4.1]{rue23} replacing $\ell_\infty^n$ with $c_0$ (and called ``$c_0$-octahedrality''). Even though $c_0$-octahedrality is strictly stronger than universal octahedrality \cite[Example 3.9]{rue23}, many examples of $c_0$-octahedral spaces were obtained such as $\Lip(M)$ or $L_1$-preduals whose norm is octahedral. 

In this paper we will perform a similar strategy for the study of $\kappa$-(rigid) octahedrality in spaces of operators. Indeed, in Section~\ref{sect:necesufi} we will obtain a necessary condition for a (subspace of a) space of operators $L(Y,X)$ to be $\kappa$-(rigid) octahedral under assumptions of uniform convexity or rotundity in Theorem~\ref{theo:condinece}. A sufficient condition is established in Theorem~\ref{theo:condisufi} which, after a brief discussion, will motivate the definition of $(\ell_\infty(\alpha),\beta)$-(rigid) octahedrality. We will prove in Theorem~\ref{theo:linftykappaoh} that if $X$ is an $(\ell_\infty(\alpha),\beta)$-(rigid) octahedral space then $L(Y,X)$ is $\kappa$ octahedral for any cardinal $\kappa$ smaller than $\min\{\alpha,\beta\}$ and for every Banach space $Y$ (when $\alpha$ is uncountable). This motivates the study of $(\ell_\infty(\alpha),\beta)$-(rigid) octahedrality in $C(K)$ spaces (Section~\ref{sect:C(K)}), in ultrapower spaces (Section~\ref{section:ultrapowers}) and in spaces of universal disposition (Section~\ref{section:espadispouni}). To point out some examples in the above mentioned sections, we prove in Theorem~\ref{theo:ultralinfi} that if a Banach space $X$ is $\ell_\infty(\omega_0,\beta)$-octahedral then $X_\mathcal U$ is $\ell_\infty(\omega_1,\beta)$-rigid octahedral for every countably incomplete ultrafilter $\mathcal U$ over any infinite set $I$. We also prove that if $\alpha$ is any infinite cardinal and $X$ is $\ell_\infty(\omega_0,\beta)$-octahedral then there exists a countably incomplete ultrafilter $\mathcal U$ over an infinite set $I$ such that $X_\mathcal U$ is $\ell_\infty(\alpha,\beta)$-rigid octahedral. Concerning Section~\ref{section:espadispouni}, we prove that if $\kappa$ is an infinite cardinal and $X$ is a space of (almost) universal disposition for density $\kappa$ then $X$ is $(\ell_\infty(\alpha),\beta)$-rigid octahedral (resp. $(\ell_\infty(\alpha),\beta)$ octahedral) if $\alpha,\beta$ are cardinals such that $\max\{2^{\alpha'},\beta'\}<\kappa$ holds for every $\alpha'<\alpha, \beta'<\beta$. In Section~\ref{sect:applications} we prove, as an application of Proposition~\ref{prop:ejenocardisupe}, that for every infinite cardinal $\kappa$ there exists a space of operators which is $\kappa$-rigid octahedral but it fails to be $\kappa^+$-octahedral (see Remark~\ref{remark:rigidohnosepara}). In particular,  for $\kappa=\omega_0$ we find the first example in the literature of Banach space which is rigid octahedral but fails to be $\omega_1$-octahedral. Finally, in Section~\ref{sect:byprodlargeocta} we apply some constructions behind the proofs of the results in Section~\ref{section:ultrapowers} to prove that, given a octahedral Banach space $X$, the following assertions hold:
\begin{itemize}
    \item $X_\mathcal U$ is $\omega_1$-rigid octahedral for every countably incomplete ultrafilter $\mathcal U$ over any infinite set $I$ (Theorem~\ref{prop:improhardtke}) and;
    \item For any infinite cardinal $\kappa$ there exists a countably incomplete ultrafilter $\mathcal U$ over some infinite set $J$ such that $X_\mathcal U$ is $\kappa$-rigid octahedral (Theorem~\ref{theo:superlargeohultrapower}).
\end{itemize}
The above mentioned results improve some results from \cite{hard18}.

\section{Notation and preliminary results}

Unless we state the contrary we will consider real Banach spaces. Given a Banach space $X$, we denote by $B_X$ and $S_X$ the closed unit ball and the unit sphere respectively. We also denote by $X^*$ the topological dual of $X$. For a given subset $E$ of $X$ we write $\lspan(E)$ (resp. $\clspan(E)$) for the (closed) linear span of $E$ and $\vert E\vert$ to denote the cardinality of the set $E$. Given two Banach spaces $X$ and $Y$, we denote by $L(X,Y)$ the space of all bounded operators from $X$ to $Y$ and $F(X,Y)$ (resp. $K(X,Y)$) the space of all finite rank and bounded (resp. compact) operators. Given an $\eps > 0$ we will say that $T \in L(X,Y)$ is an $\eps$-isometry if for any $x \in X$ we have the following inequalities
$$(1-\eps) \norm{x} \le \norm{T(x)} \le (1+\eps) \norm{x}.$$
Finally, we will say that a Banach space $X$ has density $\kappa$, denoted by $\dens(X)$, if it is the least cardinal that a subset spanning a dense subspace of $X$ can have. With this definition we recover the classical one for infinite dimensional Banach spaces, that is, the least cardinal of a dense subset of $X$, but we have the advantage that for a finite dimensional Banach space $X$ we obtain that $\dens(X) = \dim(X)$. For any other unexplied notation from Banach spaces theory we refer to \cite{fab}.

As we use some set theoretic notions about cardinals along the text, we refer the reader to \cite[Chapters 3 and 5]{jech03} where one can find more than enough information.

\subsection{Ultrapowers}

Given a Banach space $X$ and an infinite set $I$, we denote  $\ell_\infty(I,X):=\{f\colon I\longrightarrow X: \sup_{i\in I}\Vert f(i)\Vert<\infty\}$. Given a free ultrafilter $\mathcal U$ over $I$, consider $N_\mathcal U:=\{f\in \ell_\infty(I,X): \lim_\mathcal U \Vert f(i)\Vert=0\}$. The \textit{ultrapower of $X$ with respect to $\mathcal U$} is
the Banach space
$$X_\mathcal U:=\ell_\infty(I,X)/N_\mathcal U.$$
We will naturally identify a bounded function $f\colon I\longrightarrow X$ with the element $(f(i))_{i\in I}$. In this way, we denote by $[x_i]_{\mathcal U,i}$ or simply by $[x_i]_\mathcal U$, if no confusion is possible, the coset in $X_\mathcal U$ given by $(x_i)_{i\in I}+N_\mathcal U$.

From the definition of the quotient norm, it is not difficult to prove that $\Vert [x_i]_\mathcal U\Vert=\lim_\mathcal U \Vert x_i\Vert$ for every $[x_i]_\mathcal U\in X_\mathcal U$. This implies that the canonical inclusion $j:X\longrightarrow X_\mathcal U$ given by the equation
$$j(x):=[x]_\mathcal U$$
is an into linear isometry. It is known (c.f. e.g. Propositions 6.1 and 6.2 in \cite{hein80}) that $X_\mathcal U$ is finitely representable in $j(X)$. However, from an inspection of the proofs and the operators used there it can be concluded that $j(X)$ is indeed an almost isometric ideal in $X_\mathcal U$ (see definition below).

Observe that, thanks to \cite[Remark 2.4]{mrz26}, given a ultrapower space $X_\mathcal U$  for some free ultrafilter $\mathcal U$ over an infinite set $I$ and given $x\in B_{X_\mathcal U}$, we can always choose $x_i\in B_{X}$ such that $x=[x_i]$, that is, we can always select representatives where every element is in the unit ball. Moreover, an inspection in the proof reveals that the same happens if we replace the unit ball with the unit sphere. We will make use of these facts throughout the text without any further mention.

We refer the reader to \cite[Section 4.1]{sepibook} for background on ultrapowers of Banach spaces.

\subsection{Spaces of universal disposition and transfinite (almost) isometric ideals}

We say that a Banach space $X$ is \textit{of universal disposition (resp. almost universal disposition) for density $\kappa$} ($\text{(A)UD}_{<\kappa}$ in short) if for any two normed spaces $Y\subseteq Z$ with $\dens(Z)< \kappa$ and any isometry $t:Y\longrightarrow X$ there exists (for every $\varepsilon>0$) an extension $T:Z\longrightarrow X$ which is an isometry (respectively such that $(1-\varepsilon)\Vert z\Vert\le \Vert T(z)\Vert\le (1+\varepsilon)\Vert z\Vert$ holds for every $z\in Z$). We refer the reader to \cite{gk11} and references therein for background about spaces of (almost) universal disposition.

In the papers \cite{almr26,mrz25} it was proved that (A)UD$_{<\kappa}$ Banach spaces satisfy transfinite geometric properties like transfinite Daugavet, octahedrality or almost squareness. These results were obtained by providing a useful characterisation of the above spaces in terms of the notion of \textit{transfinite (almost) isometric ideal} from \cite[Definition 4.1]{mrz25}: given a Banach space $X$, a subspace $Y$ of $X$ and an infinite cardinal $\kappa$, we say that $Y$ is a \textbf{$\kappa$} (almost) isometric ideal in $X$ if given any subspace $E$ of $X$ with $\dens(E) <\kappa$ (and $\eps >0$), there exists an ($\eps$-)isometry $T: E \ra Y$ that preserves the points of $E \cap Y$. 

Observe that the above notion is a generalisation of the notion of \textit{almost isometric ideal} from \cite[Definition 1.3]{aln2}, which in the above language is nothing but $\omega_0$ almost isometric ideals. 

Observe also that a Banach space $X$ is a (A)UD$_{<\kappa}$ if, and only if, $X$ is a $\kappa$ (almost) isometric ideal in every Banach space that contains it \cite[Theorem 4.2]{mrz25}. 

The above result, far from anecdotic, reveals a powerful tool to study properties of Banach spaces in spaces of universal disposition. Indeed, if we consider any property of Banach spaces (P) satisfying the following two conditions:
\begin{enumerate}
    \item Given any Banach space $X$ there exists a Banach space $Y$ enjoying property (P) and $X\subseteq Y$ and;
    \item If $X$ is a Banach space with property (P) and $Y\subseteq X$ is a $\kappa$ (almost) isometric ideal in $X$ then $Y$ also satisfies the property (P),
\end{enumerate}
then any space of (A)UD$_{<\kappa}$ satisfies property (P). This is the way in which it was proved that if $X$ is a (A)UD$_{<\kappa}$ Banach space then $X$ has the $\kappa$-rigid Daugavet property (respectively $\kappa$-Daugavet property) in \cite[Theorem 8.4]{almr26} (the case of $\kappa=\omega_0$ was proved in \cite[Corollary 4.5]{aln2} with the same idea). We will also exploit this idea in Section~\ref{section:espadispouni}.

\section{Necessary and sufficient conditions}\label{sect:necesufi}

In this section we pursue to look for a sufficient condition on a Banach space $X$ which guarantee that $L(Y,X)$ is $<\kappa$ octahedral for every Banach space $Y$. To this end, we will begin with the following result, where we will obtain a necessary condition under the assumption that $Y$ is uniformly convex.

\begin{theorem}\label{theo:condinece}
Let $X$ and $Y$ be two Banach spaces and $\kappa$ be an uncountable cardinal. Let $H$ be a linear subspace of $L(Y,X)$ containing $F(Y,X)$. Then:
\begin{enumerate}
    \item If $H$ is $\kappa$-octahedral and $Y$ is uniformly convex, then $X$ satisfies the following condition: for every subspace  $Z\subseteq X$ with $\dens(Z)<\kappa$, every subspace $W\subseteq Y$ with $\dens(W)<\kappa$ and every $\varepsilon>0$ there exists $\phi\in B_{L(Y,X)}$ satisfying that
    $$\Vert z+\phi(w)\Vert\geq (1-\varepsilon)(\Vert z\Vert+\Vert w\Vert)$$
    holds for every $z\in Z$ and $w\in W$. In particular $\phi_{|W}$ is an $\frac{1}{1-\varepsilon}$-isometry.

    \item If $H$ $\kappa$-rigid octahedral and every $y\in S_Y$ is a strongly exposed point (in particular, such happens if $Y$ is LUR), then $X$ satisfies the following condition: for every subspace  $Z\subseteq X$ with $\dens(Z)<\kappa$ and every subspace $W\subseteq Y$ with $\dens(W)<\kappa$ there exists $\phi\in B_{L(Y,X)}$ satisfying that
    $$\Vert z+\phi(w)\Vert=\Vert z\Vert+\Vert w\Vert$$
    holds for every $z\in Z$ and $w\in W$. In particular $\phi_{|W}$ is an into isometry.
\end{enumerate}
\end{theorem}

\begin{proof}
For the proof of (1) take $\varepsilon>0$ and $Z\subseteq X$ and $W\subseteq Y$ such that $\max\{\dens(Z),\dens(W)\}<\kappa$, and select an infinite cardinal $\kappa'<\kappa$ such that $\max\{\dens(W),\dens(Z)\}\leq \kappa'$. Take a dense subset $\{z_\alpha: \alpha<\kappa'\}\subseteq Z$ of $S_Z$ and a dense subset $\{w_\alpha: \alpha<\kappa'\}\subseteq S_W$ of $S_W$. Since $Y$ is uniformly convex select $\delta>0$ with the property that if $y,z\in B_Y$ satisfy $\Vert y+z\Vert>2-\delta$ then $\Vert y-z\Vert<\frac{\varepsilon}{8}$. We can assume that $\delta<\frac{\varepsilon}{8}$.  

For every $\alpha,\beta<\kappa'$ consider the norm-one operator $T_{\alpha,\beta}:=f_\alpha\otimes z_\beta\in F(Y,X)\subseteq H$, where $f_\alpha\in S_{Y^*}$ is such that $f_\alpha(w_\alpha)=1$. Since $H$ is assumed to be $<\kappa$-octahedral we can find an operator $\phi\in S_H$ such that
$$\Vert T_{\alpha,\beta}+\phi\Vert>2-\delta\ \forall \alpha,\beta<\kappa'.$$
By the definition of the operator norm select, for every $\alpha,\beta<\kappa'$, an element $y_{\alpha,\beta}\in S_Y$ such that $\Vert T_{\alpha,\beta}(y_{\alpha,\beta})+\phi(y_{\alpha,\beta})\Vert>2-\delta$. Taking into account the definition of $T_{\alpha,\beta}$ we get that
\[\begin{split}
2-\delta<\Vert T_{\alpha,\beta}(y_{\alpha,\beta})+\phi(y_{\alpha,\beta})\Vert& =\Vert f_\alpha(y_{\alpha,\beta})z_\beta+\phi(y_{\alpha,\beta})\Vert\\
& \leq \vert f_\alpha(y_{\alpha,\beta})\vert \Vert z_\beta\Vert+\Vert\phi(y_{\alpha,\beta})\Vert\\
& \leq \vert f_\alpha(y_{\alpha,\beta})\vert+1.
\end{split}\]
Up to replacing $y_{\alpha,\beta}$ with its opposite we can assume without loss of generality that $f_\alpha(y_{\alpha,\beta})\geq 0$. Consequently, from the above inequality we infer that $f_\alpha(y_{\alpha,\beta})>1-\delta$, so
$$2-\delta<f_\alpha(y_{\alpha,\beta}+w_\alpha)\leq \Vert y_{\alpha,\beta}+w_\alpha\Vert.$$
By the condition on $\delta$ we get that $\Vert y_{\alpha,\beta}-w_\alpha\Vert<\frac{\varepsilon}{8}$. Now
\[\begin{split}
2-\delta<\Vert T_{\alpha,\beta}(y_{\alpha,\beta})+\phi(y_{\alpha,\beta})\Vert& \leq \Vert T_{\alpha,\beta}(w_{\alpha})+\phi(w_{\alpha})\Vert+\Vert T_{\alpha,\beta}+\phi\Vert\Vert y_{\alpha,\beta}-w_\alpha\Vert\\
& <\Vert T_{\alpha,\beta}(w_\alpha)+\phi(w_\alpha)\Vert+\frac{\varepsilon}{4}.
\end{split}
\]
Thus $\Vert T_{\alpha,\beta}(w_\alpha)+\phi(w_\alpha)\Vert>2-\delta-\frac{\varepsilon}{4}>2-\frac{\varepsilon}{2}$ holds for every $\alpha<\kappa'$. Consequently, given $\beta<\kappa'$, for any $\alpha<\kappa'$ we have
\[\begin{split}
2-\frac{\varepsilon}{2}<\Vert T_{\alpha,\beta}(w_\alpha)+\phi(w_\alpha)\Vert=\Vert z_\beta+\phi(w_\alpha)\Vert.    
\end{split}\]
Taking into account that $\{w_\alpha:\alpha<\kappa'\}$ is dense in $S_W$ and $\{z_\beta: \beta<\kappa'\}$ is dense in $S_Z$ we get that
$$2-\frac{\varepsilon}{2}\leq \Vert z+\phi(w)\Vert$$
holds for every $z\in S_Z, w\in S_W$. From there it is not difficult to prove that
$$\Vert z+\phi(w)\Vert\geq (1-\varepsilon)(\Vert z\Vert+\Vert w\Vert)$$
holds for every $z\in Z$ and $w\in W$.

The proof of (2) follows similar lines: select an infinite cardinal $\kappa'<\kappa$ such that $\max\{\dens(W),\dens(Z)\}\leq \kappa'$ and let $\{z_\alpha: \alpha<\kappa'\}\subseteq Z$ be a dense subset of $S_Z$ and $\{w_\alpha: \alpha<\kappa'\}\subseteq S_W$ be a dense subset of $S_W$. For every $\alpha<\kappa'$ pick a functional $f_\alpha\in S_{Y^*}$ which strongly exposes $w_\alpha$. 

As before, for every $\alpha,\beta<\kappa'$ consider the norm-one operator $T_{\alpha,\beta}:=f_\alpha\otimes z_\beta\in F(Y,X)\subseteq H$. Since $H$ is assumed to be $<\kappa$-rigid octahedral we can find an operator $\phi\in S_H$ such that
$$\Vert T_{\alpha,\beta}+\phi\Vert=2\ \forall \alpha,\beta<\kappa'.$$
Pick any $\alpha,\beta<\kappa'$. By the definition of the operator norm select a sequence $(y_n)\subseteq S_Y$ such that $\Vert T_{\alpha,\beta}(y_n)+\phi(y_n)\Vert\rightarrow 2$. This implies that
$$\vert f_\alpha(y_n)\vert=\Vert f_\alpha(y_n)z_\beta\Vert =\Vert T_{\alpha,\beta}(y_n)\Vert\rightarrow 1.$$
Assume WLOG that $f_\alpha(y_n)\geq 0$ for every $n$, so $f_\alpha(y_n)\rightarrow 1$. Since $f_\alpha$ strongly exposes $w_\alpha$ we infer that $(y_n)\rightarrow w_\alpha$ in the norm topology. Now the above convergence implies $\{\Vert T_{\alpha,\beta}(y_n)+\phi(y_n)\Vert\}\rightarrow \Vert T_{\alpha,\beta}(w_\alpha)+\phi(w_\alpha)\Vert$. By the uniqueness of limit we have
$$\Vert T_{\alpha,\beta}(w_\alpha)+\phi(w_\alpha)\Vert=2.$$
Since $\alpha,\beta<\kappa'$ were arbitrary, following the proof of (1) we infer that
$$\Vert z+\phi(w)\Vert=\Vert z\Vert+\Vert w\Vert$$
holds for every $z\in Z$ and $w\in W$, as desired.
\end{proof} 

Let us make some comments on the above result.

\begin{remark}\label{remark:condinece}
\begin{enumerate}
    \item The above result for $\kappa$ octahedrality is an extension of \cite[Lemma 3.7]{llr2}, where it was proved that if $L(Y,X)$ is octahedral and $Y$ is uniformly convex then $Y$ is finitely representable in $X$. In Theorem~\ref{theo:condinece} (1) we obtain that every subspace of $Y$ of density $<\kappa$ is $(1+\varepsilon)$ embeddable in $X$, which naturally yields a transfinite extension of finite representability.
    \item For $\kappa=\omega_0$ it is not difficult to see that (1) does remains true making of compactness of the unit sphere of finite dimensional Banach spaces and working with finite nets in the unit sphere (see for instance the proof of \cite[Lemma 3.2]{rue23}). However, for $\kappa=\omega_0$ we do not know whether (2) holds true. With our techniques, we can prove that under the asumptions of (2) then $X$ satisfies the following: for every finite set $F\subseteq X$ and every finite set $G\subseteq Y$ there exists $\phi\in B_{H}$ such that
    $$\Vert f+\phi(g)\Vert=\Vert f\Vert+\Vert g\Vert$$
    holds for every $f\in F$ and every $g\in G$. Whether we can replace the finite sets $F$ and $G$ with finite dimensional subspaces $W\subseteq Y$ and $Z\subseteq X$ is unclear to the author.
    \item Examples of Banach spaces $X$ satisfying the thesis of (2) in Theorem~\ref{theo:condinece} for $\omega_1$ are $X=L_\infty(\mu)$ for $\mu$ atomless and $X=\Lip(M)$ for any length metric space $M$. The case of $L_\infty(\mu)$ follows since $L_1(\mu)\pten Y=L_1(\mu,Y)$ has the Daugavet property for every Banach space $Y$ as $\mu$ being atomless (c.f. e.g. \cite[Theorem 3.4.4]{kmrzw25}). An application of \cite[Lemma 2.12]{kssw} yields that $(Y\pten L_1(\mu))^*=L(Y,L_1(\mu)^*)=L(Y,L_\infty(\mu))$ is $\omega_1$-rigid octahedral. The case of $X=\Lip(M)$ follows with the same argument because $\mathcal F(M)\pten Y$ has the Daugavet property for every Banach space $Y$ by \cite[Corollary 3.2]{meru25}, then the $\omega_1$-octahedrality of $(Y\pten \mathcal F(M))^*=L(Y,\mathcal F(M)^*)=L(Y,\Lip(M))$ follows again by \cite[Lemma 2.12]{kssw}.
\end{enumerate}
\end{remark}

As a direct consequence we obtain the following result.

\begin{corollary}\label{cor:notensorlargeoh}
Given any infinite-dimensional uniformly convex Banach space $Y$, it follows that $K(Y,X)$ fails to be $\omega_1$-octahedral. The same conclussion holds replacing $K(Y,X)$ with $Y\iten X$.
\end{corollary}

\begin{proof}
If there existed $X$ such that $K(Y,X)$ was $\omega_1$-octahedral, then Theorem~\ref{theo:condinece} would imply the existence of $\phi\in K(Y,X)$ such that $\phi$ is an into isomorphism, which is clearly impossible if $Y$ is infinite dimensional.
\end{proof}

A kind of converse of Theorem~\ref{theo:condinece} can be obtained for $L(Y,X)$ in the following terms. 

\begin{theorem}\label{theo:condisufi}
Let $X$ be a Banach space and $\kappa$ be an uncountable cardinal. Then:
\begin{enumerate}
\item Assume that for every Banach space $Y$ such that $\dens(Y)<\kappa$, every subset $G\subseteq X$ of cardinality smaller than $\kappa$ and every $\varepsilon>0$ there exists an operator $\phi:Y\longrightarrow X$ satisfying that 
$$\Vert z+\phi(y)\Vert\geq (1-\varepsilon)(\Vert z\Vert+\Vert y\Vert)$$
holds for every $y\in Y, z\in G$.

Then $L(Y,X)$ is $\kappa$-octahedral. 

\item Assume that for every Banach space $Y$ such that $\dens(Y)<\kappa$ and any subset $G\subseteq X$ of cardinality smaller than $\kappa$ there exists an operator $\phi:Y\longrightarrow X$ satisfying that 
$$\Vert z+\phi(y)\Vert=\Vert z\Vert+\Vert y\Vert$$
holds for every $y\in Y, z\in G$.

Then $L(Y,X)$ is $\kappa$-rigid octahedral. 
\end{enumerate}
\end{theorem}

\begin{proof}
Let us prove (2), being the proof of (1) completely similar. To this end, let $\kappa'<\kappa$ and let $\{T_\alpha:\alpha<\kappa'\}\subseteq S_{L(Y,X)}$, and let us find $T\in S_{L(Y,X)}$ such that $\Vert T_\alpha+T\Vert=2$ for every $\alpha$.

In order to do so, given $\alpha<\kappa'$ and $n\in\mathbb N$, find $y_\alpha^n\in S_X$ such that $\Vert T_\alpha(y_\alpha^n)\Vert>1-\frac{1}{n}$. Set $G:=\{T_\alpha(y_\alpha^n):n\in\mathbb N, \alpha<\kappa'\}\subseteq X$. Since the $\vert G\vert<\kappa$ we get by the assumption an operator $\phi:Y\longrightarrow X$ such that
$$\Vert z+\phi(y)\Vert=\Vert z\Vert+\Vert y\Vert$$
holds for every $y\in Y$ and $z\in G$. Now, given $\alpha<\kappa'$ and $n\in\mathbb N$ we get, in particular that
$$\Vert T_\alpha+T\Vert\geq \Vert T_\alpha(y_\alpha^n)+T(y_\alpha^n)\Vert= \Vert T_\alpha(y_\alpha^n)\Vert+\Vert y_\alpha^n\Vert=\Vert T_\alpha(y_\alpha^n)\Vert+1>2-\frac{1}{n}.$$
The arbitrariness of $n$ implies $\Vert T_\alpha+\phi\Vert=2$, and the arbitrariness of $\alpha$ shows that $T=\phi$ does the trick.
\end{proof}

Before going on the following remark is pertinent.

\begin{remark}\label{remark:condisufi}
In view of the statement in Theorem~\ref{theo:condinece} one may wonder whether we can get rid of the density condition of $\dens(Y)<\kappa$ in Theorem~\ref{theo:condisufi} by working with a suitable subspace $W\subseteq Y$ with $\dens(W)<\kappa$. Following the proof of (2), if we consider $W:=\overline{\spann}\{y_\alpha^n: n\in\mathbb N, \alpha<\kappa'\}$, then $W$ is a subspace of density $\leq \kappa'<\kappa$. However, the mere existence of an operator $\phi:W\longrightarrow X$ with the condition that
$$\Vert z+\phi(w)\Vert=\Vert z\Vert +\Vert w\Vert\ \forall z\in Z, w\in W$$
is not sufficient to guarantee that $L(Y,X)$ is $\kappa$-rigid octahedral since we would need to work with a norm-preserving extension $\Phi:Y\longrightarrow X$ of $\phi$ to define a vector in $L(Y,X)$. 

A similar problem arises when trying a similar extension of statement (1).
\end{remark}

The above remark suggests that, if we pursue to find examples of Banach spaces such that $L(Y,X)$ is $\kappa$-octahedral for every Banach spaces $Y$ taking advantages of the ideas behind the proof of Theorem~\ref{theo:condisufi}, a good idea is to request that $X$ satisfies the assumptions of this theorem for a Banach space $Y$ which is, on the one hand, universal for the class of Banach spaces of density character smaller than $\kappa$ and, on the other hand, with the property that every operator $t:Z\longrightarrow Y$ admits a norm-preserving extension $T:W\longrightarrow Y$ to any Banach space $W$ containing $Z$ (in short, that $Y$ is isometrically injective). 

The above intuition drives to the following definition.

\begin{definition}\label{def:linfikappaocta}
Let $X$ be a Banach space and let $\alpha,\beta$ be two infinite cardinals. We say that $X$ is $(\ell_\infty(\alpha),\beta)$-octahedral if every subset $Z\subseteq X$ with $\vert Z\vert<\beta$ and any $\alpha'<\alpha$ there exists an operator $\phi:\ell_\infty(\alpha')\longrightarrow X$ such that
$$\Vert z+\phi(y)\Vert\geq (1-\varepsilon)(\Vert z\Vert+\Vert y\Vert)$$
holds for every $z\in Z$ and every $y\in \ell_\infty(\alpha')$. 

If $X$ satisfies the above condition for $\varepsilon=0$ we say that $X$ is \textit{$(\ell_\infty(\alpha),\beta)$-rigid octahedral}.
\end{definition}

\begin{remark}
\begin{enumerate}
\item Recall that if $\alpha=\beta=\omega_0$ we arrive to the definition of universal octahedrality by \cite[Theorem 3.7]{rue23}.

\item Observe that if $\beta$ is uncountable or we consider $(\ell_\infty(\alpha),\omega_0)$-octahedrality, the above is equivalent to the corresponding property taking $Z$ as a subspace of $X$ with $\dens(Z)<\beta$ by easy density arguments ($\beta$ uncountable) or compactness arguments ($\beta=\omega_0$ in the non-rigid version) on $S_Z$. 
\end{enumerate}
\end{remark}

Now we have the result that we were looking for.

\begin{theorem}\label{theo:linftykappaoh}
Let $X$ be Banach space and let $\kappa,\alpha,\beta$ be infinite cardinals such that $\kappa\leq \min\{\alpha,\beta\}$. Then:
\begin{enumerate}
    \item If $X$ is $(\ell_\infty(\alpha),\beta)$-octahedral then $L(Y,X)$ is $\kappa$-octahedral for every Banach space $Y$.
    \item If $X$ is $(\ell_\infty(\alpha),\beta)$-rigid octahedral and $\alpha$ is uncountable then $L(Y,X)$ is $\kappa$-rigid octahedral for every Banach space $Y$.
\end{enumerate}
\end{theorem}

\begin{proof}
Let us prove (1). Let $Y$ be any Banach space. In order to show that $L(Y,X)$ is $\kappa$ octahedral select $\{T_\gamma: \gamma<\kappa'\}\subseteq S_{L(Y,X)}$, where $\kappa'<\kappa$, and pick $\varepsilon>0$. Let us find $T\in S_{L(Y,X)}$ satisfying that $\Vert T_\gamma+T\Vert>2(1-\varepsilon)^2$ holds for every $\gamma<\kappa'$.

Given $\gamma<\kappa'$ select an element $y_\gamma\in S_{Y}$ such that $\Vert T(y_\gamma)\Vert>1-\varepsilon$. Consider $V:=\overline{\spann\{y_\gamma: \gamma<\kappa'\}}\subseteq Y$. Observe that since $\dens(V)\leq \kappa'$ there exists $\alpha'<\alpha$ and a norm-one operator $j:V\longrightarrow \ell_\infty(\alpha')$ such that 
$$\Vert j(v)\Vert\geq (1-\varepsilon)\Vert v\Vert\ \forall v\in V.$$
Indeed, we have two possibilities:
\begin{enumerate}
    \item If $\alpha$ is uncountable, then there exists an infinite cardinal $\alpha'$ such that $\kappa'\leq \alpha'<\alpha$, and then the result follows since $\ell_\infty(\alpha')$ contains an isometric copy of every Banach space of density $\leq \alpha'$ (so in this particular case $j$ can be even chosen to be an into linear isometry).
    \item If $\alpha=\omega_0$ then $\kappa'$ is finite, so $V$ is finite dimensional. Since every Banach space is finitely representable in $c_0$ there exists a bounded operator $k:V\longrightarrow c_0$ with $\Vert k\Vert\leq 1$ and $\Vert k(v)\Vert\geq \left(1-\frac{\varepsilon}{2} \right)\Vert v\Vert$ holds for every $v\in V$. Then taking $\alpha'=n$ large enough then the operator $j:=P_n\circ k$, where $P_n:=c_0\longrightarrow \ell_\infty^n$ is the natural projection on the first $n$ coordinates, does the trick.
\end{enumerate}
Because of the isometric injectivity of $\ell_\infty(\alpha')$ we can find a norm-one operator $P:Y\longrightarrow \ell_\infty(\alpha')$ extending $j$. 

Define $Z:=\{T(y_\gamma): \gamma<\kappa'\}\subseteq X$, whose cardinality is strictly smaller than $\kappa\leq \beta$. By the assumption there exists an operator $\phi:\ell_\infty(\alpha')\longrightarrow X$ such that
$$\Vert z+\phi(y)\Vert\geq (1-\varepsilon)(\Vert z\Vert+\Vert y\Vert)$$
holds for every $z\in Z$ and every $y\in \ell_\infty(\alpha')$.

Define $T:=\phi\circ P:Y\longrightarrow X$, which is a norm-one operator. Let us prove that $\Vert T_\gamma+T\Vert>2-\varepsilon$ holds for every $\gamma<\kappa'$. Pick $\gamma<\kappa'$. Taking into account the assumptions on $j$ we have
\[
\begin{split}
\Vert T_\gamma+T\Vert\geq \Vert T_\gamma (y_\gamma)+\phi(P(y_\gamma))\Vert& \geq (1-\varepsilon)(\Vert T_\gamma(y_\gamma)\Vert+\Vert P(y_\gamma)\Vert)\\
& >(1-\varepsilon)\left(1-\varepsilon+\Vert j(y_\gamma)\Vert \right)\\
& =(1-\varepsilon)\left(2-2\varepsilon\right),
\end{split}
\]
as desired.

The proof of (2) is completely similar taking into account that the above mapping $j$ can be always be taken isometric in virtue of the uncountability assumption on $\alpha$.
\end{proof}

In the next sections we present several examples of Banach spaces enjoying $(\ell_\infty(\kappa),\beta)$-(rigid) octahedrality.

\section{$C(K)$ spaces}\label{sect:C(K)}

In this section we will obtain some results $C(K)$ spaces. In order to do so, let us start recalling the following definition from \cite[Definition 4.3]{almr26}: given a compact space $K$, the cardinal number known as \emph{reaping number} of $K$, denoted by $\mathfrak{r}(K)$, is the least cardinality of a family $\mathcal{F}$ of nonempty open subsets of $K$ such that for every two disjoint closed sets $L_1$ and $L_2$ there exists $W\in\mathcal{F}$  such that either $W\cap L_1 = \emptyset$ or $W\cap L_2=\emptyset$.

The above cardinal is introduced and systematically studied in \cite{almr26}, where it is proved to be a generalisation of the concept of the reaping number of a Boolean algebra (via the Stone duality mapping) and it turns out to characterise transfinite versions of the Daugavet property.

An inspection in the proof of \cite[Theorem 4.5]{almr26} reveals that (taking $g=0$ there) if $\mathfrak r(K)\geq \kappa$ for an uncountable cardinal $\kappa$ then, given any $\kappa'<\kappa$ and any family $\{f_\alpha: \alpha<\kappa'\}\subseteq S_{C(K)}$ there exists a sequence of pairwise disjoint closed sets with non-empty interior $L_n^+,L_n^-$ and functions $\psi_n^{\pm}$ such that $\psi_n^{\pm}=1$ on $L_n^{\pm}$ and $\psi_n^{\pm}=0$ on the rest of sets $L_m^{\pm}$ and on $L_n^{\mp}$, and such that
$$\Vert f_\alpha\pm (\psi_n^+-\psi_n^-)\Vert=2.$$
Consequently, taking $g_n:=\psi_n^+-\psi_n^-$, it follows that $(g_n)\subseteq \{h\in S_{C(K)}: \Vert f_\alpha\pm h\Vert=2\ \forall \alpha<\kappa\}$ and that $(g_n)$ is a sequence of functions whose supports are pairwise disjoint. Consequently, we have the following result.

\begin{proposition}\label{prop:C(K)normal}
Let $\kappa$ be an uncountable cardinal and $K$ be a compact Hausdorff space such that $\mathfrak r(K)\geq \kappa$ (in other words, $C(K)$ is $\kappa$-rigid octahedral). Then, for every $\{f_\alpha:\alpha<\kappa'\}\subseteq S_{C(K)}$ there exists a bounded operator $T:c_0\longrightarrow C(K)$ satisfying that
$$\Vert f_\alpha+T(g)\Vert=1+\Vert g\Vert$$
holds for every $\alpha<\kappa'$ and $g\in c_0$. In particular, $C(K)$ is $(\ell_\infty(\omega_0),\kappa)$-rigid octahedral.
\end{proposition}

Our next aim is to prove that in the case that $K$ is extremally disconnected then we can replace $c_0$ with $\ell_\infty$ in the above result. In such case $L_n^\pm$ can be taken to be clopen and thus $\psi_n^\pm=\chi_{L_n^\pm}$. Now the result follows from the following result.

Before going on in the next result, we need the following definition from \cite[Section 17 Q]{willard}: a compact space $K$ is said to be \textit{projective} if for every pair of compact spaces $L,M$, for every continuous function $f:K\longrightarrow L$ and every continuous and onto function $g:M\longrightarrow L$ there always exists a continuous function $h:K\longrightarrow M$ such that $f=g\circ h$. Recall that a compact space $K$ is projective if, and only if, $K$ is extremally disconnected \cite[Section 17 Q]{willard}.

Taking the above characterisation into account let us provide the following proposition, whose proof was provided to the author by Antonio Avil\'es.

\begin{proposition}\label{prop:linfextremadisco}
Let $K$ be an extremally disconnected compact space and let $\{C_i: i\in I\}$ be a family of clopen and pairwise disjoint sets. Then there exists an into linear isometry $\phi:\ell_\infty(I)\longrightarrow C(K)$ such that $\phi(e_i)=\chi_{C_i}$.
\end{proposition}

\begin{proof}
Since $K$ is extremally disconnected then $K$ is projective.

We denote by $I\cup\{\infty\}$ the one point compactification of the set $I$ endowed with the discrete topology. Now the universal extension property of $\beta I$, the Stone-Cech compactification of $I$ endowed with the discrete topology, implies the existence of a continuous function $q:\beta I\longrightarrow I\cup\{\infty\}$ such that $q(i)=i$ holds for every $i\in I$. Observe that $e\notin I$ implies $q(e)=\infty$ (in particular $q$ is onto). In order to prove it, select a net $\{i_\alpha\}\rightarrow e$ such that $i_\alpha\in I$ holds for ever $\alpha$ (this is from the density of $I$ in $\beta I$). Now, given $i\in I$ we can find $\alpha_0$ such that $i_\alpha\neq i$ for $\alpha>\alpha_0$. This implies that $q(i_\alpha)=i_\alpha\neq i$ holds for every $\alpha\geq \alpha_0$. Since the topology on $I$ is the discrete topology we must have $q(e)=\lim q(i_\alpha)\neq i$. The arbitrariness of $i\in I$ forces $q(e)=\infty$, as desired.

Define the mapping $f:K\longrightarrow I\cup\{\infty\}$ by the equation
$$f(x):=\left\{\begin{array}{cc}
i     & \mbox{ if }x\in C_i \\
\infty     & \mbox{ otherwise.}
\end{array} \right.$$
It is immediate that $f$ is continuous at every point of $\bigcup\limits_{i\in I} C_i$. On the other hand, let $x\in K\setminus \bigcup\limits_{i\in I} C_i$ and let $U$ be an open set in $I\cup\{\infty\}$ containing $\infty$. We can assume with no loss of generality that $U=I\setminus C$, where $C\subseteq I$ is compact. Since $I$ is discrete we obtain that $C$ is finite, say $C=\{i_1,\ldots, i_p\}$. Now, since $x\notin \bigcup\limits_{i\in I} C_i$ we can get an open set $O$ containing $x$ in $K$ such that $O\cap C_{i_j}=\empty$ for $1\leq j\leq p$. It is immediate that $O\subseteq f^{-1}(I\setminus C)=f^{-1}(U)$. 

Now $f$ and $q$ are continuous and $q$ is onto. From the projectivity of the compact space $K$ we can find a continuous mapping $\hat f:K\longrightarrow \beta I$ such that $q\circ \hat f=f$. Observe that $\hat f$ is onto since $f$ is onto too and $f(x)=i$ holds for every $x\in C_i$, $i\in I$. Now, since $\hat f$ is onto we get that the composition operator $C_{\hat f}:C(\beta I)=\ell_\infty(I)\longrightarrow C(K)$ defined by
$$C_{\hat f}(g)=g\circ \hat f$$
is an into linear isometry. Finally, observe that given $i\in I$ we have $C_{\hat f}(e_i)=\chi_{C_i}$. Indeed, given $x\in C_i$ we have
$$C_{\hat f}(e_i)(x)=(e_i\circ \hat f)(x)=e_i(i)=1,$$
whereas if $x\notin C_i$ then $f(x)\neq i\in \beta I$ and thus
$$C_{\hat f}(e_i)(x)=e_i(f(x))=0,$$
and we are done.
\end{proof}

A direct application of Proposition~\ref{prop:C(K)normal} together with Proposition~\ref{prop:linfextremadisco} yields the following result.

\begin{theorem}\label{theo:C(K)extremadisconnected}
Let $K$ be compact Hausdorff topological space such that $\mathfrak r(K)\geq \kappa$ (in other words, $C(K)$ is $\kappa$-rigid octahedral). If $K$ is extremally disconnected then $C(K)$ is $(\ell_\infty(\omega_1),\kappa)$-rigid octahedral.
\end{theorem}

\begin{remark}
Observe that in \cite[Remark 5.10]{almr26} it is proved that, for any cardinal $\kappa$, there exists a totally disconnected compact space $K$ such that $C(K)$ is $\kappa$-rigid octahedral. As a consequence of Theorem~\ref{theo:C(K)extremadisconnected}, for every cardinal $\kappa$ there are $C(K)$ spaces which are $(\ell_\infty(\omega_1),\kappa)$-rigid octahedral.
\end{remark}

\section{Ultrapowers}\label{section:ultrapowers}

In this section we pursue to analyse how $(\ell_\infty(\alpha),\beta)$ octahedrality is preserved by taking ultrapower spaces

\begin{theorem}\label{theo:ultralinfi}
Let $X$ be a Banach space and let $\kappa$ be an infinite cardinal. Assume that $X$ is $(\ell_\infty(\omega_0),\kappa)$ octahedral. Let $\mathcal U$ be any countably incomplete ultrafilter over any infinite set $I$. Then $X_\mathcal U$ is $(\ell_\infty(\omega_1),\kappa)$-rigid octahedral.
\end{theorem}

\begin{proof} In order to prove the result let $k'<\kappa$ and let a subset $Z=\{[z_i^\alpha]: \alpha<\kappa'\}\subseteq X_\mathcal U$.
Since $\mathcal U$ is countably incomplete we can consider a decreasing sequence $\{A_n: n\in\mathbb N\}\subseteq \mathcal U$ such that $\bigcap\limits_{n=1}^\infty A_n=\emptyset$. We can assume with no loss of generality that $A_1=I$. Let $n\in\mathbb N$. Given $i\in A_n\setminus A_{n+1}$ we can find by the assumption a bounded operator $\phi_i:\ell_\infty^n\longrightarrow X$ satisfying that
$$\Vert z_i^\alpha+\phi_i((\lambda_1,\ldots, \lambda_n))\Vert\geq \left(1-\frac{1}{n}\right)(\Vert z_i^\alpha\Vert+\Vert (\lambda_1,\ldots,\lambda_n)\Vert_\infty)$$
holds for every $\alpha<\kappa'$ and every $(\lambda_1,\ldots, \lambda_n)\in \ell_\infty^n$. 

Now, given $i\in A_n\setminus A_{n+1}$ define $P_i:\ell_\infty\longrightarrow \ell_\infty^n$ the natural projection
$$P_i((\lambda_k)):=(\lambda_1,\ldots,\lambda_n).$$
Finally define $\Phi:\ell_\infty\longrightarrow X_\mathcal U$ by the equation
$$\Phi(\lambda)(i):=\phi_i(P_i(\lambda)).$$
$\Phi$ is well defined, linear and bounded with $\Vert \Phi\Vert\leq 1$. Let us prove that $\Phi$ satisfies the requirements.

In order to do so select any $\alpha<\kappa'$ and $\lambda=(\lambda_n)\in \ell_\infty$, and let us prove that
$$\Vert [z_i^\alpha]+\Phi(\lambda)\Vert=\Vert [z_i^\alpha]\Vert+\Vert\lambda\Vert.$$
In order to do so we can assume with no loss of generality that $\lambda\neq 0$ (for $\lambda=0$ the requested equality trivially follows). 

Pick $\varepsilon>0$ and let $n\in\mathbb N$ such that $\vert \lambda_n\vert>\Vert \lambda\Vert_\infty-\varepsilon$. Now set
$$B_1:=\{i\in I: \vert \Vert[z_i^\alpha]\Vert-\Vert z_i^\alpha\Vert\vert<\varepsilon\}\in\mathcal U.$$
Take also
$$B_2:=\{i\in I: \left\vert \Vert [z^\alpha
_i]+\Phi(\lambda)\Vert-\Vert z_i^\alpha+\Phi(\lambda)(i)\Vert\right\vert<\varepsilon\}\in\mathcal U.$$
Now define $C:=B_1\cap B_2\cap A_k\in\mathcal U$, where $k$ is big enough to guarantee $\frac{1}{k}<\varepsilon$ and $k\geq n$. Given $j\in C$ we have
\[\begin{split}
\Vert [z_i^\alpha]+\Phi(\lambda)\Vert& \mathop{>}\limits^{j\in B_2}\Vert z_j^\alpha+\Phi(\lambda)(j)\Vert -\varepsilon\mathop{=}\limits^{j\in A_k}\Vert z_j^\alpha+\phi_k(P_k(\lambda))\Vert-\varepsilon \\
& \geq \left(1-\frac{1}{k}\right)(\Vert z_j^\alpha\Vert+\Vert P_k(\lambda)\Vert)-\varepsilon\\
& \mathop{>}\limits^{\frac{1}{k}<\varepsilon} \left(1-\varepsilon\right)(\Vert z_j^\alpha\Vert+\Vert P_k(\lambda)\Vert)-\varepsilon\\
& \mathop{>}\limits^{j\in B_1}(1-\varepsilon)(\Vert [z_i^\alpha]\Vert-\varepsilon+\Vert (\lambda_1\ldots, \lambda_k)\Vert_\infty)-\varepsilon\\
& \geq (1-\varepsilon)(\Vert [z_i^\alpha]\Vert-\varepsilon+\vert\lambda_n\vert)-\varepsilon\\
& >(1-\varepsilon)(\Vert [z_i^\alpha]\Vert-\varepsilon+\Vert \lambda\Vert_\infty-\varepsilon)-\varepsilon.
\end{split}\]
The arbitrariness of $\varepsilon$ concludes the proof.
\end{proof}

In particular the above proves that if $K(Y,X)$ is $\kappa$ octahedral for every finite-dimensional Banach space $Y$ then $L(Y,X_\mathcal U)$ is $\kappa$-rigid octahedral for every separable Banach space $Y$. In the sequel we will show that we can even get $(\ell_\infty(\alpha),\kappa)$-octahedrality for any $\alpha$ if we select a  big enough ultrapower space.

\begin{theorem}\label{theo:ultrapotlargos}
Let $\kappa$ be an infinite cardinal and assume that $X$ is a $(\ell_\infty(\omega_0),\kappa)$ octahedral Banach space. Then, given any uncountable cardinal $\alpha$, there exists a countably incomplete ultrafilter $\mathcal U$ over some infinite set $J$ such that $X_\mathcal U$ is $(\ell_\infty(\alpha),\kappa)$-rigid octahedral.
\end{theorem}

\begin{proof}
Let $I$ be any set with $\vert I\vert=\alpha$ and let $J$ be the set of finite subsets of $I$. Given $A\in J$ call
$$F_A:=\{B\in J: A\subseteq B\}\subseteq J.$$
Observe that $\mathcal B:=\{F_A: A\in J\}\subseteq \mathcal P(J)$ defines a filter basis on $J$. Select $\mathcal U$ to be a free ultrafilter containing the filter basis $\mathcal B$. It is immediate that $\mathcal U$ is countably incomplete. Indeed, if we select a sequence $\{A_n: n\in\mathbb N\}\subseteq J$ such that $\bigcup\limits_{n=1}^\infty A_n$ is infinite, then $\bigcap\limits_{n=1}^\infty F_{A_n}=\emptyset$. Taking it into account define a decreasing sequence $\{U_n:n\in\mathbb N\}\subseteq \mathcal U$ with $U_1=J$ and $\bigcap\limits_{n\in\mathbb N} U_n=\emptyset$. This implies that $J=\bigcup\limits_{n\in\mathbb N} U_n\setminus U_{n+1}$ defines a partition on $J$.

In order to prove that $X_\mathcal U$ satisfies our requirement select $\kappa'<\kappa$ and a collection $\{[z_i^\alpha]: \alpha<\kappa'\}\subseteq X_\mathcal U$.

Given $A\in J$ define the natural projection $P_A:\ell_\infty(I)\longrightarrow \ell_\infty(A)$. 

Now, given $A\in U_n\setminus U_{n+1}$ we have, by the assumption, a bounded operator $\phi_A:\ell_\infty(A)\longrightarrow X$ satisfying that
\begin{equation}\label{theo:ultralargaortocondi}\Vert z_A^\alpha+\phi_A(g)\Vert\geq \left(1-\frac{1}{n}\right)(\Vert z_A^\alpha\Vert+\Vert g\Vert)
\end{equation}
holds for every $\alpha<\kappa'$ and every $g\in \ell_\infty(A)$. Now define $\Phi:\ell_\infty(I)\longrightarrow X_\mathcal U$ by the equation
$$\Phi(f)(A):=\phi_A(P_A(f)).$$
Similar arguments to the proof of Theorem~\ref{theo:ultralinfi} allows to conclude the result. Indeed, select any $\alpha<\kappa'$ and $g\in \ell_\infty(I)\setminus\{0\}$, and let us prove that
$$\Vert [z_A^\alpha]+\Phi(g)\Vert=\Vert [z_A^\alpha]\Vert+\Vert g\Vert.$$
In order to do so select $\varepsilon>0$ and define the sets
$$B_1:=\{F\in J: \vert \Vert [z_A^\alpha]\Vert-\Vert z_F^\alpha\Vert \vert<\varepsilon\}\in\mathcal U.$$
$$B_2:=\{F\in J: \vert \Vert [z_A^\alpha]+\Phi(g)\Vert - \Vert z_F^\alpha+\phi_F(P_F(g)) \Vert \vert<\varepsilon\}\in\mathcal U.$$
Observe also that there exists some finite set $G\subseteq I$ such that $\Vert P_G(g)\Vert>\Vert g\Vert-\varepsilon$, and clearly $\Vert P_F(g)\Vert>\Vert g\Vert-\varepsilon$ if $G\subseteq F$. Now, given $n\in\mathbb N$, select any $F\in B_1\cap B_2\cap F_G\cap A_n$. Then we have
\[\begin{split}
    \Vert [z_A^\alpha]+\Phi(g)\Vert & \mathop{\geq}\limits^{F\in B_2} \Vert z_F^\alpha+\phi_F(P_F(g)) \Vert-\varepsilon\\
    & \mathop{\geq}\limits^{\mbox{\tiny{\eqref{theo:ultralargaortocondi}}}} \left(1-\frac{1}{n}\right)(\Vert z_F^\alpha\Vert+\Vert P_F(g)\Vert)-\varepsilon\\
    & \mathop{\geq}\limits^{F\in B_1} \left(1-\frac{1}{n}\right)(\Vert [z_A^\alpha]\Vert-\varepsilon+\Vert P_F(g)\Vert)-\varepsilon\\
    & \mathop{\geq}^{G\subseteq F} \left(1-\frac{1}{n}\right)(\Vert [z_A^\alpha]\Vert+\Vert g\Vert-2\varepsilon)-\varepsilon.
\end{split}\]
The arbitrariness of $n$ and $\varepsilon$ finishes the proof.
\end{proof}

\section{Spaces of universal disposition}\label{section:espadispouni}

In this section we pursue to show that spaces of universal disposition enjoy $(\ell_\infty(\alpha),\beta)$ octahedrality properties. In order to do so, let us begin with the following result, which shows that $(\ell_\infty(\alpha),\beta)$ octahedrality is inherited by transfinite almost isometric ideals.

\begin{proposition}\label{prop:heretransaii}
Let $X$ be a Banach space which is $(\ell_\infty(\alpha),\beta)$-octahedral. Let $\kappa$ be an infinite cardinal with $\max\{2^{\alpha'},\beta'\}<\kappa$ for every $\alpha'<\alpha$ and every $\beta'<\beta$. If $Y\subseteq X$ is a $\kappa$-ai ideal in $X$ then $Y$ is $(\ell_\infty(\alpha),\beta)$-octahedral.
\end{proposition}

\begin{proof}
Let $\alpha'<\alpha, \beta'<\beta$, $\{y_i: i<\beta'\}\subseteq S_Y$ be a subset of $S_X$ and $\varepsilon>0$. Since $S_Y\subseteq S_X$ and $X$ is $(\ell_\infty(\alpha),\beta)$-octahedral we can find a bounded operator $\phi:\ell_\infty(\alpha')\longrightarrow X$ with $\Vert \phi\Vert\leq 1$ and 
$$\Vert y_i+\phi(\lambda)\Vert\geq (1-\varepsilon)(1+\Vert \lambda\Vert)$$
holds for every $i<\beta'$ and every $\lambda\in \ell_\infty(\alpha')$. 

Set $Z:=\overline{\spann}(\{y_i: i<\beta'\}\cup \phi(\ell_\infty(\alpha')))\subseteq X$. Observe that $\dens(Z)\leq \max\{\beta',\dens(\ell_\infty(\alpha'))\}=\max\{\beta',2^{\alpha'}\}<\kappa$. Since $Y$ is a $\kappa$ ai ideal in $X$ we can find a bounded operator $T:Z\longrightarrow Y$ with the properties that $T(z)=z$ holds for $z\in Z\cap Y$ (in particular $T(y_i)=y_i$ for every $i<\beta'$) and 
$$(1-\varepsilon)\Vert z\Vert\leq \Vert T(z)\Vert\leq (1+\varepsilon)\Vert z\Vert$$
holds for every $z\in Z$. 

If we take $T\circ\phi:\ell_\infty(\alpha')\longrightarrow Y$ we have that $\Vert T\circ\phi\Vert\leq 1+\varepsilon$ and, given $i<\beta'$ and $\lambda\in \ell_\infty(\alpha')$, we have
\[\begin{split}
\Vert y_i+T(\phi(\lambda))\Vert=\Vert T(y_i)+T(\phi(\lambda))\Vert& =\Vert T(y_i+\phi(\lambda))\Vert\geq (1-\varepsilon)\Vert y_i+\phi(\lambda)\Vert\\
& \geq (1-\varepsilon)^2(1+\Vert\lambda\Vert)
\end{split}\]
Now let $\Phi:=\frac{T\circ\phi}{1+\varepsilon}$, which satisfies $\Vert \Phi\Vert\leq 1$. Moreover, given $i<\beta'$ and $\lambda\in\ell_\infty(\alpha')$, we have
\[\begin{split}
\Vert y_i+\Phi(\lambda)\Vert& =\left\Vert y_i+T\left(\phi\left(\frac{\lambda}{1+\varepsilon}\right)\right)\right\Vert\geq (1-\varepsilon)^2\left(1+\frac{\Vert \lambda\Vert}{1+\varepsilon}\right)\\
& \geq (1-\varepsilon)^2\left(\frac{1}{1+\varepsilon}+\frac{\Vert \lambda\Vert}{1+\varepsilon}\right )=\frac{(1-\varepsilon)^2}{1+\varepsilon}(1+\Vert \lambda\Vert).
\end{split}\]
The arbitrariness of $\varepsilon$ concludes the proof.
\end{proof}

Observe that with a similar (indeed, a simpler) proof we can obtain a rigid version in the following lines.

\begin{proposition}\label{prop:heretransiso}
Let $X$ be a Banach space which is $(\ell_\infty(\alpha),\beta)$-rigid octahedral. Let $\kappa$ be an infinite cardinal with $\max\{2^{\alpha'},\beta'\}<\kappa$ holds for every $\alpha'<\alpha, \beta'<\beta$. If $Y\subseteq X$ is a $\kappa$- isometric ideal in $X$ then $Y$ is $(\ell_\infty(\alpha),\beta)$-rigid octahedral.
\end{proposition}

\begin{remark}
Observe that in the above result, the requirement $\max\{2^{\alpha'},\beta'\}<\kappa$ for every $\alpha'<\alpha$ and $\beta'<\beta$ replaces the (at first glance) more natural requirement $\max\{2^\alpha,\beta\}\leq \kappa'$ in order to avoid problems when $\alpha=\omega_0$. 
\end{remark}

Now we can get the following result.

\begin{theorem}
Let $X$ be a Banach space and let $\kappa$ be an infinite cardinal. Then:
\begin{enumerate}
    \item If $X$ is a space of almost universal disposition for spaces of density $<\kappa$, then $X$ is $(\ell_\infty(\alpha),\beta)$-octahedral for any $\alpha,\beta$ such that $\max\{2^{\alpha'},\beta'\}<\kappa$ holds for every $\alpha'<\alpha, \beta'<\beta$. 

    \item If $X$ is a space of universal disposition for spaces of density $<\kappa$, then $X$ is $(\ell_\infty(\alpha),\beta)$-octahedral for any $\alpha,\beta$ such that $\max\{2^{\alpha'},\beta'\}<\kappa$ holds for every $\alpha'<\alpha, \beta'<\beta$. 
\end{enumerate}
\end{theorem}

\begin{proof} Let $\alpha,\beta$ be two infinite cardinals such that $\max\{2^\alpha,\beta\}<\kappa$. An inspection in the proof of \cite[Theorem 8.4]{almr26} reveals that $X$ embeds in a $C(K)$ space which is $\beta$ octahedral, so $C(K)$ is $(\ell_\infty(\omega_0),\beta)$ octahedral by Proposition~\ref{prop:C(K)normal}.


In virtue of Theorem~\ref{theo:ultrapotlargos} we can find a free ultrafilter on a sufficiently big set $I$ such that $C(K)_\mathcal U$ is $(\ell_\infty(\alpha),\beta)$-rigid octahedral. As $X$ embeds isometrically in $C(K)$ then it embeds isometrically in $C(K)_\mathcal U$. 

On the one hand, if $X$ is a space of almost universal disposition with respect to spaces of density $<\kappa$, then $X$ is a $\kappa$ ai-ideal in $C(K)_\mathcal U$ in virtue of \cite[Theorem 4.2]{mrz25}, and then $X$ is $(\ell_\infty(\kappa),\beta)$-octahedral by an application Proposition~\ref{prop:heretransaii}, obtaining (1). If, on the other hand, $X$ is a space of universal disposition with respect to spaces of density $<\kappa$, then $X$ is a $\kappa$ isometric ideal in $C(K)_\mathcal U$ in virtue of \cite[Theorem 4.2]{mrz25}, and then $X$ is $(\ell_\infty(\kappa),\beta)$-rigid octahedral by an application Proposition~\ref{prop:heretransiso}, obtaining (2) and concluding the theorem.
\end{proof}

\section{Applications}\label{sect:applications}

In this section we will apply all the above discussion in order to distinguish different notions of octahedrality in spaces of operators. In order to do so, given an uncountable cardinal $\kappa$ let us denote by 
$$L(Y,X)_\kappa:=\{T:Y\longrightarrow X: \dens(T(Y))<\kappa\}\subseteq L(Y,X).$$
We will denote $L(Y,X)_{\omega_0}:=\overline{F(Y,X)}$.

Observe that $L(Y,X)_\kappa$ is a closed subspace of $L(Y,X)$ for any infinite cardinal $\kappa$ (hence complete) and that the following conditions hold:
\begin{enumerate}
    \item If $T\in L(Y,X)_\kappa$ and $S\in L(Z,Y)$ then $T\circ S\in L(Z,X)_\kappa$ and,
    \item If $T\in L(Y,X)_\kappa$ and $R\in L(X,W)$ then $R\circ T\in L(Y,W)_\kappa$. 
\end{enumerate}

Now we have the following result:

\begin{proposition}\label{prop:ejenocardisupe}
Let $\kappa$ be an infinite cardinal, let $X$ be any $(\ell_\infty(\kappa^+),\kappa)$-rigid octahedral space. Then $L(\ell_2(\kappa^+),X)_\kappa$ is $\kappa$-rigid octahedral but fails to be $\kappa^+$-octahedral.
\end{proposition}

\begin{proof}
The failure of $\kappa^+$-octahedrality follows from a literal application of Theorem~\ref{theo:condinece}. Let us prove that $L(\ell_2(\kappa^+),X)$ is $\kappa$-rigid octahedrality. In order to do so, select $\kappa'<\kappa$ and $\{T_\alpha:\alpha<\kappa'\}\subseteq S_{L(Y,X)_\kappa}$, and let us find an operator $T:Y\longrightarrow X$ in $L(Y,X)_\kappa$ such that
$$\Vert T_\alpha+T\Vert=2$$
holds for every $\alpha<\kappa'$. In order to do so, given $n\in\mathbb N$, find $y_\alpha^n\in S_{\ell_2(\kappa^+)}$ such that $\Vert T(y_\alpha^n)\Vert>1-\frac{1}{n}$. Define $Z:=\overline{\spann}\{y_\alpha^n: \alpha<\kappa', n\in\mathbb N\}$, which is a subspace of $\ell_2(\kappa^+)$ of density at most $\max\{\kappa',\omega_0\}\leq \kappa$. Consequently, there exists an isometric linear embedding $\psi:Z\longrightarrow \ell_\infty(\kappa)$. Since $X$ is assumed to be $(\ell_\infty(\kappa^+),\kappa)$ octahedral we can find a norm-one operator $\Phi:\ell_\infty(\kappa)\longrightarrow X$ such that
$$\Vert w+\Phi(\lambda)\Vert=\Vert w\Vert+\Vert \lambda\Vert$$
holds for every $w\in W:=\overline{\spann}\{T_\alpha(y_\alpha^n): n\in\mathbb N, \alpha<\kappa'\}\subseteq X$ and every $\lambda\in \ell_\infty(\alpha)$. Finally, since every closed subspace of $\ell_2(\kappa^+)$ is $1$ complemented consider $P:\ell_2(\kappa^+)\longrightarrow Z$ to be a norm-one operator such that $P(z)=z$ holds for every $z\in Z$. Now $T=\Phi\circ \psi\circ P:\ell_2(\kappa^+)\longrightarrow X$ is the desired operator. Indeed, $T\in L(\ell_2(\kappa^+),X)_\kappa$ since $\psi\in L(Z,\ell_\infty(\kappa))_\kappa$ and $\Vert T\Vert=1$. Moreover, given $\alpha<\kappa'$ and $n\in\mathbb N$, we get
\[\begin{split}
\Vert T_\alpha+T\Vert\geq \Vert T_\alpha(y_\alpha^n)+\Phi(\psi(P(y_\alpha^n)))\Vert & > 1-\frac{1}{n}+\Vert \psi(P(y_\alpha^n))\Vert\\
& =1-\frac{1}{n}+\Vert P(y_\alpha^n)\Vert\\
& >2\left(1-\frac{1}{n}\right).
\end{split}\]
The arbitrariness of $n$ implies $\Vert T_\alpha+T\Vert=2$, and the proof is finished.
\end{proof}

Let us end with the following remark.

\begin{remark}\label{remark:rigidohnosepara}
\begin{enumerate}
    \item The above example for $\kappa=\omega_0$ shows that $\ell_2\iten X$ is $\omega_0$-rigid octahedral but fails $\omega_1$ octahedrality. Such an example, up to the best of our knowledge, was not previously known in the literature.
    \item The result above holds replacing $\ell_2(\kappa^+)$ with $\ell_p(\kappa^+)$ for $1<p<\infty$ because such spaces satisfy that given any subspace $Z\subseteq \ell_p(\kappa^+)$ with $\dens(Z)<\kappa^+$  there exists a 1-complemented subspace $Y\subseteq \ell_p(\kappa^+)$ with $\dens(Z)=\dens(Y)$ and $Z\subseteq Y$.
\end{enumerate}
\end{remark}



\section{A byproduct: transfinite octahedrality in ultrapowers}\label{sect:byprodlargeocta}

In this section we will take advantage of the techniques of Section~\ref{section:ultrapowers} to provide a couple of results which show how ultrapower spaces naturally stregthen the transfinite octahedrality of the underlying space and, in particular, we will obtain an improvement of a result by J. D. Hardtke \cite[Theorem 4.1]{hard18}.

Indeed, \cite[Theorem 4.1]{hard18} establishes that, given a Banach space $X$, then $X$ is octahedral if, and only if, $X_\mathcal U$ is rigid octahedral for every free ultrafilter over $\mathbb N$. With a bit of extra work, we will obtain $\omega_1$-rigid octahedrality on $X_\mathcal U$ in the next result.

\begin{proposition}\label{prop:improhardtke}
Let $X$ be a Banach space. The following are equivalent:
\begin{enumerate}
    \item $X$ is octahedral.
    \item $X_\mathcal U$ is $\omega_1$-rigid octahedral for every countably incomplete ultrafilter $\mathcal U$ over any infinite set $I$.
    \item $X_\mathcal U$ is $\omega_1$-rigid octahedral for every free ultrafilter $\mathcal U$ over $\mathbb N$.
    \item $X_\mathcal U$ is rigid octahedral for every countably incomplete ultrafilter $\mathcal U$ over any infinite set $I$.
    \item $X_\mathcal U$ is rigid octahedral for every free ultrafilter $\mathcal U$ over $\mathbb N$.
\end{enumerate}
\end{proposition}

\begin{proof}
It is immediate that (2) implies (3), (4) and (5), and (5)$\Rightarrow$(1) was proved by J.~D.~Hardtke in \cite[Theorem 4.1]{hard18}, so let us prove (1)$\Rightarrow$(2).

In order to do so let $I$ be any infinite set, $\mathcal U$ be any countably incomplete ultrafilter $\mathcal U$ over $I$, and let us prove that $X_\mathcal U$ is $\omega_1$-rigid octahedral. In order to do so, let $\{[x_i^n]: i\in I, n\in\mathbb N\}\subseteq S_X$, and let us find $[x_i]\in S_{X_\mathcal U}$ such that $\Vert [x_i^n]+[x_i]\Vert=2$ holds for every $n\in\mathbb N$. Up to a change of representative we can assume WLOG that $\Vert x_i^n\Vert=1$ holds for every $i\in I$ and every $n\in\mathbb N$ \cite[Remark 2.4]{mrz26}. 

Recall that since $\mathcal U$ is countably incomplete we can assume that there exists a decreasing sequence $\{A_n\}\subseteq \mathcal U$ such that $\bigcap\limits_{n\in\mathbb N} A_n=\emptyset$. We can assume as well that $A_1=I$, and then $I=\bigcup\limits_{n\in \mathbb N} A_n\setminus A_{n+1}$ forms a partition of the set $I$.

It is time to define $[x_i]$. To this end, given $i\in I$, since $\{A_n\setminus A_{n+1}: n\in\mathbb N\}$ is a partition of $I$ there exists a unique $n\in\mathbb N$ such that $i\in A_n\setminus A_{n+1}$. Now, since $X$ is octahedral we can select $x_i\in S_X$ satisfying that
$$\Vert x_i^k+x_i\Vert>2-\frac{1}{n}\ \forall 1\leq k\leq n.$$
Observe that $[x_i]\in S_{X_\mathcal U}$. Let us prove that the above element satisfies our requirements. In order to do so, select any $m\in \mathbb N$ and $\varepsilon>0$. Now the set
$$A:=\{i\in I: \vert \Vert [x_i^m]+[x_i]\Vert-\Vert x_i^m+x_i\Vert\vert<\varepsilon\}$$
is an element of $\mathcal U$. Next, select $n\geq m$ such that $\frac{1}{n}<\varepsilon$, and observe that $A\cap A_n\in\mathcal U$ and, in particular, it is non-empty. Thus, given $i\in A\cap A_n$ we get
\[\begin{split}
\Vert [x_i^m]+[x_i]\Vert\mathop{>}\limits^{i\in A}\Vert x_i^m+x_i\Vert-\varepsilon\mathop{>}\limits^{i\in A_n}2-\frac{1}{n}-\varepsilon\mathop{>}^{\frac{1}{n}<\varepsilon} 2-2\varepsilon.    
\end{split}\]
The arbitrariness of $\varepsilon$ implies $\Vert [x_i^m]+[x_i]\Vert=2$, and the arbitrariness of $m\in\mathbb N$ concludes the proof.
\end{proof}

In the same spirit as in Theorem~\ref{theo:ultrapotlargos}, the next result says that we can get stronger notions of octahedrality in ultrapower spaces if we construct a sufficiently large ultrapower space.

\begin{theorem}\label{theo:superlargeohultrapower}
Let $X$ be an octahedral Banach space and $\kappa$ be any infinite cardinal. Then there exists an infinite set $J$ and a countably incomplete ultrafilter $\mathcal U$ over $J$ such that $X_\mathcal U$ is $\kappa$-rigid octahedral.
\end{theorem}

\begin{proof}
Let $J$ be the set of finite subsets of $\kappa$ and consider any free ultrafilter $\mathcal U$ extending the filter basis
$$\mathcal B:=\{F_A: A\in J\},$$
where $F_A:=\{B\in J: A\subseteq B\}$. We know that $\mathcal U$ is a countably incomplete ultrafilter over $J$ as in the proof of Theorem~\ref{theo:ultrapotlargos}. Let us prove that $\mathcal U$ satisfies the requirements. In order to do so, select $\{[x_F^\alpha]: \alpha<\kappa'\}\subseteq S_{X_\mathcal U}$, where $\kappa'<\kappa$, and let us find $[x_F]\in S_{X_\mathcal U}$  such that $\Vert [x_F^\alpha]+[x_F]\Vert=2$ holds for every $\alpha$. In order to do so, let us define $x_F$.

Since $\mathcal U$ is countably incomplete we can find a decreasing sequence $\{B_n: n\in\mathbb N\}\subseteq \mathcal U$ such that $\bigcap\limits_{n\in\mathbb N} B_n=\emptyset$. Thus $\{B_n\setminus B_{n+1}: n\in\mathbb N\}$ is a partition of the set $J$. With this, we will define the desired element $[x_F]$. Given $F\in J=\bigcup\limits_{n\in\mathbb N} B_n\setminus B_{n+1}$ there exists a unique $n\in\mathbb N$ such that $F\in B_n\setminus B_{n+1}$. Now since $F\subseteq \kappa$ is a finite set, we find by the octahedrality of $X$ an element $x_F\in S_X$ with the property that
$$\Vert x_F^\alpha+x_F\Vert>2-\frac{1}{n}\ \forall \alpha\in F.$$
Let us prove that $[x_F]\in S_{X_\mathcal U}$ satisfies the desired requirements. In order to do so select $\alpha<\kappa$ and $\varepsilon>0$. Now set
$$A:=\{F\in J: \vert \Vert [x_F^\alpha]+[x_F]\Vert-\Vert x_F^\alpha+x_F\Vert \vert<\varepsilon\}\in \mathcal U.$$
Now take any $F\in A\cap B_n\cap F_{\{\alpha\}}\in\mathcal U$, where $n$ is taken large enough to satisfy $\frac{1}{n}<\varepsilon$. Now we have
\[\begin{split}
\Vert [x_F^\alpha]+[x_F]\Vert & \mathop{>}^{F\in A}\Vert x_F^\alpha+x_F\Vert-\varepsilon\mathop{>}^{\alpha\in F\in B_n}2-\frac{1}{n}-\varepsilon>2-2\varepsilon.
\end{split}
\]
Since $\varepsilon>0$ was arbitrary we infer $\Vert [x_F^\alpha]+[x_F]\Vert=2$, and the arbitrariness of $\alpha<\kappa'$ concludes the proof.
\end{proof}

\section*{Acknowledgements}  The author is deeply grateful to Antonio Avil\'es for providing the proof of Proposition~\ref{prop:linfextremadisco} and for further fruitful conversations. The author also thanks Miguel Mart\'in for fruitful conversations.

This research has been supported  by MCIU/AEI/FEDER/UE\\  Grant PID2021-122126NB-C31 and by Junta de Andaluc\'{\i}a Grant FQM-0185.

\section*{AI disclosure statement}

The author declares that he refused the use of any AI tool at any of the stages of the paper. Consequently, the author takes full responsibility in the results, the proofs and the language editing of this manuscript.

\end{document}